\documentclass{amsart}

\newtheorem{theorem}{Theorem}[section]
\newtheorem{lemma}[theorem]{Lemma}

\theoremstyle{definition}
\newtheorem{definition}[theorem]{Definition}

\newtheorem{corollary}[theorem]{Corollary}

\theoremstyle{remark}

\numberwithin{equation}{section}

\usepackage{graphicx}

\begin{document}

\title{Spectra of the lower triangular matrix $\mathbb{B}(r_1,\dots , r_l; s_1, \dots, s_{l'})$ over $c_0$}


\author{Sanjay Kumar mahto}
\address{Department of Mathematics, Indian Institute of Technology Kharagpur, Kharagpur 721302, India}
\curraddr{}
\email{kumarmahtosanjay@iitkgp.ac.in, skmahto0777@gmail.com}
\thanks{}

\author{Arnab Patra}
\address{Department of Mathematics, Indian Institute of Technology Kharagpur, Kharagpur 721302, India}
\curraddr{}
\email{arnptr91@gmail.com}
\thanks{}

\author{P. D. Srivastava}
\address{Department of Mathematics, Indian Institute of Technology Kharagpur, Kharagpur 721302, India}
\curraddr{}
\email{ pds@maths.iitkgp.ac.in}
\thanks{}



\date{}

\dedicatory{}

\begin{abstract}
The spectra and fine spectra of the lower triangular matrix $\mathbb{B}$ $(r_1,\dots , r_l;$ $ s_1, \dots, s_{l'})$ over the sequence space $c_0$ are determined. The diagonal and sub-diagonal entries of the matrix consist of two oscillatory sequences $r=(r_{k (\text{mod} \ l)+1})$ and $s= (s_{k(\text{mod} \ l')+1})$ respectively, whereas the rest of the entries of the matrix are zero. In particular, the spectra and fine spectra of the lower triangular matrix $\mathbb{B}(r_1,\dots , r_4; s_1, \dots, s_{6})$ over $c_0$ are discussed.
\end{abstract}

\maketitle

\section{Introduction}
Spectral theory plays an important role in the study of functional analysis, classical quantum mechanics etc. Many properties of a linear operator, for example differential operator, matrix operator etc., can be discussed by knowing the spectra of the operator.\par
Study of fine spectra of an infinite matrix over a sequence space has received much attention in recent years. Wenger \cite{wenger1975fine} has determined the fine spectra of the H$\ddot{\text{o}}$lder summability operators over the space of convergent sequences $c$. The fine spectra for weighted mean operators are determined by Rhoades in \cite{rhoades1983fine} and \cite{rhoades1989fine}. Gonzalez \cite{gonzalez1985fine} has computed the fine spectra of the Ces$\grave{\text{a}}$ro operator in the space of $p$-absolutely convergent sequences $l_p(1<p<\infty)$, whereas Reade \cite{reade1985spectrum} computed the spectrum of Ces$\grave{\text{a}}$ro operator in the space of null sequences $c_0$. The spectrum of Rhaly operators is determined by Yildirim over $c_0$ and $c$ in \cite{yildirim1996spectrum}.
 Akhmedov and Ba\c{s}ar have examined the fine spectrum of Ces$\grave{\text{a}}$ro operator and difference operator $\Delta$ over $bv_p(1\leq p < \infty)$ in \cite{akhmedov2008fine} and \cite{akhmedov2007fine} respectively. Altay and Ba\c{s}ar have considered the difference operators $\Delta$ and $B(r,s)$ over $c_0$ and $c$ in \cite{altay2004fine} and \cite{altay2005fine} respectively, whereas Furkan and Bilgi\c{c} have studied $B(r,s)$ over $l_p$ and $bv_p$ in \cite{bilgicc2008fine}. The spectrum and fine spectrum of the operators $\Delta_v$ and $\Delta_{uv}$ over $l_1$ have been investigated by Srivastava and Kumar in \cite{srivastava2012fine} and \cite{srivastava2012fine2} respectively. Panigrahi  and Srivastava determined the spectrum and fine spectrum of the operators $\Delta ^2_{uv}$ \cite{panigrahi2012spectrum} and $\Delta ^{2}_{uvw}$ \cite{panigrahi2012spectrum2} over $c_0$ and $l_1$ respectively. Recently, Birbonshi and Srivastava \cite{birbonshi2017some} and Patra et al. \cite{patra2017some} have made a study on the fine spectra of $n$th band triangular matrices for $n>1$.\par
 In this paper, we have studied the spectra and fine spectra of the generalized difference matrix $\mathbb{B}(r_1,\dots , r_l; s_1, \dots, s_{l'})$ in which the diagonal entries consist of a sequence whose terms are oscillating between the points $r_1, r_2, \dots , r_l$ and the sub-diagonal entries consist of an oscillatory sequence whose terms are oscillating between the points $s_1, s_2, \dots , s_{l'}$. Furthermore, the spectra and fine spectra of the matrix $\mathbb{B}(r_1,\dots , r_4; s_1, \dots, s_{6})$ are also discussed.\par
This paper is organized as follows: in section 2, some fundamental results are given. The section 3 is devoted to the study of the spectrum and the Goldberg's classification of the spectrum for the generalized difference matrix $\mathbb{B}(r_1,\dots , r_l; s_1, \dots, s_{l'})$. In section 4, the study of spectra and fine spectra of the matrix $\mathbb{B}(r_1,\dots , r_4; s_1, \dots, s_{6})$ are also done.

\section{Preliminaries}
Let $X$ and $Y$ be Banach spaces and $T: X \rightarrow Y$ be a bounded linear operator. The Banach space of all bounded linear operators on $X$ into $Y$ is denoted by $B(X,Y)$. We denote $B(X,X)$ by $B(X)$. The adjoint $T^*$ of a bounded linear operator $T\in B(X)$ is a bounded linear operator on the dual $X^*$ of $X$. That is, $T^* \in B(X^*)$ and is defined as $(T^* f)(x)=f(Tx)$ for all $f\in X^*$ and $x\in X$.\par
Let $X \neq \{0\}$ be a complex normed space and $T:D(T)\rightarrow X$ be a linear operator defined over $D(T) \subseteq X$. We associate an operator $T_{\lambda}$ with $T$ by $T_{\lambda} = T-\lambda I$, where $\lambda$ is a complex number and $I$ is the identity operator on $D(T)$. The resolvent operator of $T$ is denoted by $R_{\lambda}(T)$ and is defined as $R_{\lambda}(T)=T_{\lambda}^{-1}=(T-\lambda I)^{-1}$.
\begin{definition}(Regular value)\cite{kreyszig1989introductory}
A complex number $\lambda$ is said to be a regular value of a linear operator $T: D(T) \rightarrow X$ if and only if the following conditions are satisfied:\\
  (R1) $ T_{\lambda}^{-1}$ exists\\
  (R2) $T_{\lambda}^{-1}$ is bounded\\
  (R3) $T_{\lambda}^{-1}$ is defined on a set which is dense in $X$.
  \end{definition}

  \noindent The set of all regular values of the linear operator $T$ is called \emph{resolvent set} and is denoted by $\rho (T)$. The complement $\sigma(T) = \mathbb{C} - \rho(T)$ is called the \emph{spectrum} of $T$ and the members of the spectrum are called \emph{spectral values} of $T$. The spectrum $\sigma(T)$ is further partitioned into the following three disjoint sets:\\
(i)The \emph{point spectrum (discrete spectrum)} is the set $\sigma_{p}(T)$ of all complex numbers for which the resolvent operator $T_{\lambda}^{-1}$ does not exist. The members of the point spectrum are called \emph{eigen values} of $T$.\\
(ii)The \emph{continuous spectrum} is the set $\sigma_c(T)$ of all complex numbers for which the properties (R1) and (R3) satisfy but the property (R2) does not satisfy.\\
(iii)The \emph{residual spectrum} is the set $\sigma_{r}(T)$ of all complex numbers for which the property (R1) satisfy but the property (R3) does not satisfy.

Let $R(T_{\lambda})$ denotes the range of the operator $T_{\lambda}$. Goldberg \cite{goldberg2006unbounded} has further classified the spectrum using the following six properties of $R(T_{\lambda})$ and $T_{\lambda}^{-1}$:\\
(I) $R(T_{\lambda}) = X$\\
(II) $R(T_{\lambda}) \neq X$ but $\overline{R(T_{\lambda})} = X$\\
(III) $\overline{R(T_{\lambda})} \neq X$\\
and\\
(1) $T_{\lambda}^{-1}$ exists and is bounded\\
(2) $T_{\lambda}^{-1}$ exists but is not bounded\\
(3) $T_{\lambda}^{-1}$ does not exist.\\
Combining the above properties, the Goldberg's classifications of the spectrum are given in the Table \ref{table1}.

\begin{table}
\begin{center}
\begin{tabular}{c|c c c}
& I & II & III\\
\hline
1& $\rho(T,X)$ &--- & $\sigma_r(T,X)$\\
2& $\sigma_c(T,X)$ & $\sigma_c(T,X)$& $\sigma_r(T,X)$\\
3& $\sigma_p(T,X)$ &$\sigma_p(T,X)$ & $\sigma_p(T,X)$\\
\end{tabular}
\caption{Subdivisions of spectrum of a bounded linear operator}\label{table1}
\end{center}
\end{table}

\begin{theorem}\cite{taylor1957general}\label{th_adjointontobounded}
Let $T$ be a bounded linear operator on a normed linear space $X$, then $T$ has a bounded inverse if and only if $T^*$ is onto.
\end{theorem}

\begin{lemma}\cite{stieglitz1977matrixtransformationen}\label{lemma_b(c0;c0)}
An infinite matrix $A=(a_{nk})$ gives rise to a bounded linear operator $T\in B(c_0)$ from $c_0$ to itself if and only if\\
(i) each row of $A$ belongs to the space $l_1$ and supremum of their $l_1$ norms is bounded.\\
(ii) each column of $A$ belongs to $c_0$.
The operator norm of $T$ is the supremum of the $l_1$ norms of the rows.
\end{lemma}
\par
Throughout the paper, we denote the set of natural numbers by $\mathbb{N}$, the set of complex numbers by $\mathbb{C}$ and $\mathbb{N}_0 = \mathbb{N} \cup \{0\}$. We assume that $x_{-n} = 0$ for all $n\in \mathbb{N}$.

\section{Main results}
 Let $l$ and $l'$ be two natural numbers. Suppose that $L$ is the least common multiple of $l$ and $l'$. Then the matrix $\mathbb{B}(r_1,\dots , r_l; s_1, \dots, s_{l'})$ is defined as $\mathbb{B} = (b_{ij})_{i,j\geq 0}$, where
\begin{equation}
b_{ij} = \begin{cases} r_{j(\text{mod} \ l)+1} & \text{when} \  i=j,\\
 s_{j(\text{mod} \ l')+1} & \text{when} \  i=j+1,\\
 0  & \text{otherwise}.\end{cases}
\end{equation}
That is,
\begin{align*}
B&=
\left[
\begin{array}{ccccc}
r_1  & 0 & 0 & 0 & \cdots\\
s_1 & \ddots & 0 & 0 & \cdots\\
0&s_2 & r_l & 0 & \cdots\\
0&0 & \ddots & \ddots & \cdots\\
0& 0& 0&  s_{l'}& \ddots\\
\vdots&\vdots &\vdots &\vdots & \ddots\\
\end{array}
\right].
\end{align*}
If the matrix $\mathbb{B}$ transforms a sequence $x=(x_k)$ into another sequence $y = (y_k)$, then
\begin{align}
y_k &= \sum\limits_{j=0}^{\infty}b_{kj}x_j \nonumber\\
       & = b_{k,k-1}x_{k-1} + b_{k,k}x_k \nonumber \\
&= s_{(k-1)(\text{mod} \ l')+1}x_{k-1} + r_{k(\text {mod} \ l)+1}x_k.
\end{align}
\begin{corollary}
The matrix $\mathbb{B}: c_0 \rightarrow c_0$ is a bounded linear operator and $\lVert B\rVert_{(c_0 : c_0)} \leq \max\limits_{i,j}\{\lvert r_i \rvert + \lvert s_j \lvert; 1\leq i \leq l, 1 \leq j \leq l'\}$.
\end{corollary}
\par
Suppose that $a$ is an integer and $n$ is a natural number. Then, by $a(\text{mod} \ n)$, we mean the least positive integer $x$ such that $n$ divides $a-x$. We denote the set of all such nonnegative integers $x$ such that $n$ divides $a-x$ by $[a_n]$.
Let $\sigma$ and $\sigma'$ be two mappings defined on the set of integers such that
  \begin{equation*}
  \sigma(k)= k(\text{mod} \ l)+1
  \end{equation*}
  and
  \begin{equation*}
   \sigma'(k)= k(\text{mod} \ l')+1
  \end{equation*}
  respectively. Without loss of generality, we assume that $s_{\sigma'(k)} s_{\sigma'(k+1)} \cdots s_{\sigma'{(k+j)}} = 1$ and $(r_{\sigma(k)} - \lambda)(r_{\sigma(k+1)}-\lambda) \cdots (r_{\sigma(k+j)}-\lambda) = 1$ when $k+j < k$. If $\lambda$ is a complex number such that $(\mathbb{B} - \lambda I)^{-1}$ exists, then the entries of the matrix $(\mathbb{B} - \lambda I)^{-1}=(z_{nk})$ are given by
  \begin{align}\label{inverse}
   z_{nk}
  &=\begin{cases}\frac{(-1)^{n-k} s_{\sigma'(k)} \cdots s_{\sigma'(k+\zeta''-1)}}{(r_{\sigma(k)}-\lambda)\cdots (r_{\sigma(k+\zeta')}-\lambda)}\frac{(s_1 \dots s_{l'})^{m''}}{\{ (r_1 - \lambda)\cdots (r_l - \lambda)\}^{m'}}\left\{ \frac{(s_1 \cdots s_{l'})^{\frac{L}{l'}}}{ \{(r_1 - \lambda)\cdots (r_l - \lambda)\}^{\frac{L}{l}}}\right\}^m &  \text{when} \ n\geq k \\
  \frac{1}{r_k - \lambda} & \text{when} \ n = k \\
   0 & \text{otherwise}, \end{cases}
  \end{align}
  where $\zeta, \zeta'$ and $\zeta''$ are the least nonnegative integers such that
  \begin{align*}
  n-k &= m L  + \zeta\\
  \zeta&=  m' l+ \zeta' \\
  \zeta  &= m''l' + \zeta''
  \end{align*}
  for some nonnegative integers $m, m'$ and $m''$.

\begin{lemma}\label{lemma_inverse}
If $\lambda$ is a complex number such that  $(\vert \lambda - r_1\vert \cdots \vert \lambda - r_l\vert)^{{1}/{l}} > (\vert s_1 \vert \cdots \vert s_{l'}\vert)^{{1}/{l'}}$, then $(\mathbb{
B}-\lambda I)^{-1} \in B(c_0)$.
\end{lemma}
\begin{proof}
Since $s_1, s_2, \dots, s_{l'}$ are nonzero and $(\vert \lambda - r_1\vert \cdots \vert \lambda - r_l\vert)^{{1}/{l}} > (\vert s_1 \vert \cdots \vert s_{l'}\vert)^{{1}/{l'}}$, therefore $\lambda \neq r_1$, $\lambda \neq r_2, \dots ,\lambda \neq r_l$. Consequently, the matrix $\mathbb{B}-\lambda I$ is a triangle and hence $(\mathbb{B}-\lambda I)^{-1} =(z_{nk})$ exists which is given by \eqref{inverse}.
   We first consider an arbitrary row of $(\mathbb{B}-\lambda I)^{-1}$ that is a multiple of $L$, that is $n=  \widetilde{m}L$ for some $\widetilde{m}\in \mathbb{N}_0$. Now, let $k=  \hat{m}L$ for $\hat{m} = 0, 1, \dots, \widetilde{m}$. Then $n-k = (\widetilde{m}- \hat{m})L$ and $\zeta = \zeta' = \zeta''=0$. Thus, from \eqref{inverse}, we have
  \begin{equation*}
  z_{nk} = \frac{(-1)^{n-k}}{r_{\sigma(k)}-\lambda}\left\{ \frac{(s_1 \dots s_{l'})^{\frac{L}{l'}}}{\{(r_1 - \lambda)\dots (r_l - \lambda)\}^{\frac{L}{l}}}\right\}^{\widetilde{m}-\hat{m}}
  \end{equation*}
  for all $\hat{m} = 0, 1, \dots, \widetilde{m}$. Therefore,
  \begin{equation*}
  \sum\limits_{k\in [0_{L}]}\vert z_{nk}\vert = \frac{1}{\vert r_{\sigma(k)}-\lambda\vert}\sum\limits_{j=0}^{\tilde{m}}\left\{ \frac{(\vert s_1\vert \cdots \vert s_{l'}\vert)^{\frac{L}{l'}}}{\{\vert r_1 - \lambda \vert \cdots \vert r_l - \lambda \vert\}^{\frac{L}{l}}}\right\}^j,
  \end{equation*}
  where $[0_L]$ denotes the set of all nonnegative integers which are multiple of $L$. For the same row, if we consider $k=\hat{m}L+1$ for $\hat{m}= 0,1, \dots \tilde{m}-1$, then $n-k=(\widetilde{m}-\hat{m}-1)L+L-1$. Let $m_1$ and $\widetilde{m}_1$ be quotients and $\zeta'_1$ and $\zeta''_1$ be remainders when $L-1$ is divided by $l$ and $l'$ respectively, that is
  \begin{align*}
  L-1 &= m_1l + \zeta'_1\\
  L-1 &=  \widetilde{m}_1 l'+ \zeta''_1.
  \end{align*}
  Then, from \eqref{inverse}, we get
 \begin{align*}
 z_{nk} &= \frac{(-1)^{n-k} s_{\sigma'(k)} \cdots s_{\sigma'(k+\zeta''_1-1)}}{(r_{\sigma(k)}-\lambda)\cdots (r_{\sigma(k+\zeta'_1)}-\lambda)}\frac{(s_1 \cdots s_{l'})^{\widetilde{m}_1}}{\{ (r_1 - \lambda)\cdots (r_l - \lambda)\}^{m_1}}
\left\{ \frac{(s_1 \cdots s_{l'})^{\frac{L}{l'}}}{ \{(r_1 - \lambda)\cdots (r_l - \lambda)\}^{\frac{L}{l}}}\right\}^{ \widetilde{m}-\hat{m}-1}
 \end{align*}
 for all $\hat{m}= 0, 1, \dots, \widetilde{m}-1$. Hence,
 \begin{align*}
 \sum\limits_{k\in [1_L]}\vert z_{nk} \vert &=  \frac{\vert s_{\sigma'(k)}\vert \cdots \vert s_{\sigma'(k+\zeta''_1-1)} \vert}{\vert r_{\sigma(k)}-\lambda \vert \cdots  \vert r_{\sigma(k+\zeta'_1)}-\lambda \vert }  \frac{(\vert s_1 \vert \cdots \vert s_{l'} \vert )^{\widetilde{m}_1}}{\{ \vert r_1 - \lambda\vert \cdots \vert r_l - \lambda \vert \}^{m_1}}\nonumber\\
  & \qquad \times\sum\limits_{j=0}^{\widetilde{m}-1} \left\{ \frac{(\vert s_1 \vert \cdots \vert s_{l'} \vert )^{\frac{L}{l'}}}{ \{\vert r_1 - \lambda\vert \cdots \vert r_l - \lambda \vert\}^{\frac{L}{l}}}\right\}^j,
 \end{align*}
 where $[1_L]$ denotes the set of all nonnegative integers $x$ such that $L$ divides $1-x$. Similarly, for $k=\hat{m}L+2, \dots ,\hat{m}L+L-1$, we get
  \begin{align*}
 \sum\limits_{k\in [2_L]}\vert z_{nk} \vert &=  \frac{\vert s_{\sigma'(k)}\vert \cdots \vert s_{\sigma'(k+\zeta''_2-1)} \vert}{\vert r_{\sigma(k)}-\lambda \vert \cdots  \vert r_{\sigma(k+\zeta'_2)}-\lambda \vert }  \frac{(\vert s_1 \vert \dots \vert s_{l'} \vert )^{\widetilde{m}_2}}{\{ \vert r_1 - \lambda\vert \cdots \vert r_l - \lambda \vert \}^{m_2}}  \nonumber \\
  &\qquad \times\sum\limits_{j=0}^{\widetilde{m}-1} \left\{ \frac{(\vert s_1 \vert \cdots \vert s_{l'} \vert )^{\frac{L}{l'}}}{ \{\vert r_1 - \lambda\vert \cdots \vert r_l - \lambda \vert\}^{\frac{L}{l}}}\right\}^j,
 \end{align*}
 $$\vdots$$
  \begin{align*}
 \sum\limits_{k\in [(L-1)_L]}\vert z_{nk} \vert &=  \frac{\vert s_{\sigma'(k)}\vert \cdots \vert s_{\sigma'(k+\zeta''_{L-1}-1)} \vert}{\vert r_{\sigma(k)}-\lambda \vert \cdots  \vert r_{\sigma(k+\zeta'_{L-1})}-\lambda \vert }  \frac{(\vert s_1 \vert \cdots \vert s_{l'} \vert )^{\widetilde{m}_{L-1}}}{\{ \vert r_1 - \lambda\vert \cdots \vert r_l - \lambda \vert \}^{{m}_{L-1}}}\nonumber\\
   &\qquad \times \sum\limits_{j=0}^{\widetilde{m}-1} \left\{ \frac{(\vert s_1 \vert \cdots \vert s_{l'} \vert )^{\frac{L}{l'}}}{ \{\vert r_1 - \lambda\vert \cdots \vert r_l - \lambda \vert\}^{\frac{L}{l}}}\right\}^j.
 \end{align*}
 Thus,
 \begin{align}\label{eq_inverse}
 \sum\limits_{k=0}^{\infty} \vert z_{nk}\vert &= \frac{1}{\vert r_{\sigma(k)}-\lambda\vert}\sum\limits_{j=0}^{\widetilde{m}}\left\{ \frac{(\vert s_1\vert \cdots \vert s_{l'}\vert)^{\frac{L}{l'}}}{\{\vert r_1 - \lambda \vert \cdots \vert r_l - \lambda \vert\}^{\frac{L}{l}}}\right\}^j
 + M \sum\limits_{j=0}^{\widetilde{m}-1} \left\{ \frac{(\vert s_1 \vert \cdots \vert s_{l'} \vert )^{\frac{L}{l'}}}{ \{\vert r_1 - \lambda\vert \cdots \vert r_l - \lambda \vert\}^{\frac{L}{l}}}\right\}^j,
 \end{align}
 where
 \begin{align*}
 M =&  \frac{\vert s_{\sigma'(k)}\vert \cdots \vert s_{\sigma'(k+\zeta''_1-1)} \vert}{\vert r_{\sigma(k)}-\lambda \vert \cdots  \vert r_{\sigma(k+\zeta'_1)}-\lambda \vert }  \frac{(\vert s_1 \vert \cdots \vert s_{l'} \vert )^{\widetilde{m}_1}}{\{ \vert r_1 - \lambda\vert \cdots \vert r_l - \lambda \vert \}^{m_1}} \\
 &+ \dots +\frac{\vert s_{\sigma'(k)}\vert \cdots \vert s_{\sigma'(k+\zeta''_{L-1}-1)} \vert}{\vert r_{\sigma(k)}-\lambda \vert \cdots  \vert r_{\sigma(k+\zeta'_{L-1})}-\lambda \vert }  \frac{(\vert s_1 \vert \cdots \vert s_{l'} \vert )^{\widetilde{m}_{L-1}}}{\{ \vert r_1 - \lambda\vert \cdots \vert r_l - \lambda \vert \}^{{m}_{L-1}}}.
 \end{align*}

 Let $M_0 = \max\{\frac{1}{\vert r_{\sigma(k)}-\lambda\vert}, M\}$, then
 \begin{equation*}
 \sum\limits_{k=0}^{\infty} \vert z_{nk}\vert \leq  \frac{2 M_0(\vert r_1 - \lambda \vert \vert r_2 - \lambda \vert \cdots \vert r_l - \lambda \vert)^{\frac{L}{l}}}{(\vert r_1 \vert \vert r_2\vert \dots \vert r_l \vert)^{\frac{L}{l}} -  (\vert s_1\vert \vert s_2\vert \cdots \vert s_{l'}\vert)^{\frac{L}{l'}}}.
 \end{equation*}
 Therefore, $\sup_{n\in [0_L]} \sum_{k=0}^{\infty} \vert z_{nk}\vert < \infty$. Similarly, we prove that $\sup_{n\in [1_L]} \sum_{k=0}^{\infty} \vert z_{nk}\vert < \infty$, $\sup_{n\in [2_L]} \sum_{k=0}^{\infty} \vert z_{nk}\vert < \infty$, $\dots$, $ \sup_{n\in [(L-1)_L]} \sum_{k=0}^{\infty} \vert z_{nk}\vert < \infty$. Thus,
 \begin{equation*}
 \sup\limits_n \sum_{k=0}^{\infty} \vert z_{nk}\vert = \max \left \{ \sup_{n\in [0_L]} \sum_{k=0}^{\infty} \vert z_{nk}\vert, \sup_{n\in [1_L]} \sum_{k=0}^{\infty} \vert z_{nk}\vert, \dots, \sup_{n\in [(L-1)_L]} \sum_{k=0}^{\infty} \vert z_{nk}\vert\right\}.
 \end{equation*}
 This implies that $\sup_n \sum_{k=0}^{\infty} \vert z_{nk}\vert < \infty$. Like wise, for an arbitrary column of $(\mathbb{B}-\lambda I)^{-1}$, adding the entries separately whose rows $n$ belongs to $[0_L], [1_L], \dots, [(L-1)_L]$ respectively, we get $\sum_{n=0}^{\infty} \vert z_{nk}\vert<\infty$. Therefore, $\lim_{n\rightarrow \infty} \vert z_{nk} \vert $ $=0$ for all $k\in \mathbb{N}_0$. Hence, by Lemma \ref{lemma_b(c0;c0)}, the matrix $(\mathbb{B}-\lambda I)^{-1} \in B(c_0)$.
\end{proof}

Consider the set $S= \left\{\lambda \in \mathbb{C}: (\vert \lambda - r_1\vert \cdots \vert \lambda - r_l\vert)^{\frac{1}{l}} \leq (\vert s_1 \vert \cdots \vert s_{l'}\vert)^{\frac{1}{l'}}\right\}$. Then, we have the following theorem:
\begin{theorem}\label{th_spectrum}
 $ \sigma(\mathbb{B}, c_0) = S$.
\end{theorem}
\begin{proof}
 First, we prove that $ \sigma(\mathbb{B}, c_0) \subseteq S$. Let $\lambda$ be a complex number that does not belong to $S$. Then, $(\vert \lambda - r_1\vert \cdots \vert \lambda - r_l\vert)^{{1}/{l}} > (\vert s_1 \vert \cdots \vert s_{l'}\vert)^{{1}/{l'}}$.
 Therefore from Lemma \ref{lemma_inverse}, we have $(\mathbb{B}-\lambda I)^{-1} \in B(c_0)$. That is, $\lambda \notin \sigma (\mathbb{B}, c_0)$. Hence, $\sigma (\mathbb{B}, c_0) \subseteq A$. \par
 Next, we show that $S \subseteq \sigma (\mathbb{B}, c_0)$. Let $\lambda \in S$. Then, $(\vert \lambda - r_1 \vert \cdots \vert \lambda- r_l \vert)^{1/l}\leq(\vert s_1 \vert \cdots \vert s_{l'} \vert)^{1/l'}$. If $\lambda$ equals any of the $r_i$ for $i=1, \dots, l$, then the range of the operator is not dense and thus $\lambda \in \sigma (\mathbb{B}, c_0)$. Therefore, suppose $\lambda \neq r_1, \dots, \lambda \neq r_l$. Then, $\mathbb{B}-\lambda I$ is a triangle and $(\mathbb{B}- \lambda I)^ {-1} =(z_{nk})$ exists which is given by \eqref{inverse}. Let $y = (1, 0, 0, \dots) \in c_0$ and let $x = (x_k)$ be the sequence such that $(\mathbb{B}- \lambda I)^ {-1} y = x$. Then, from \eqref{inverse}, we get
 \begin{equation}
 x_{nL} = z_{nL,0}= \frac{1}{r_1 - \lambda}\left\{\frac{(s_1s_2 \dots s_{l'})^{\frac{L}{l'}}}{\{ (r_1-\lambda)\dots (r_l -\lambda)\}^{\frac{L}{l}}} \right\}^{n}.
 \end{equation}
 Since $\{ (r_1 - \lambda ) \dots (r_l - \lambda)\}^ {1/l} \leq (s_1 \dots s_{l'})^{\frac{1}{l'}}$, the subsequence $(x_{nL})$ of $x$ does not converge to $0$. Consequently, the sequence $x=(x_k)\notin c_0$. Thus, $\lambda \in \sigma (\mathbb{B}, c_0)$ and therefore $S \subseteq \sigma (\mathbb{B}, c_0)$. This proves the theorem.
\end{proof}

\begin{theorem}\label{th_pointspectrum}
$\sigma_p(\mathbb{B},c_0)=\phi$.
\end{theorem}
\begin{proof}
Let $\lambda \in \sigma_p(\mathbb{B},c_0)$. Then there exists a nonzero sequence $x=(x_k)$ such that $\mathbb{B}x=\lambda x$. This implies that
\begin{equation}\label{eq_bx=lambdax}
s_{(k-1)(\text{mod} \ l')+1}x_{k-1}+r_{k(\text{mod}\ l)+1}x_k = \lambda x_k.
\end{equation}
Let $x_{k_0}$ be the first non-zero term of the sequence $x=(x_k)$. Then from the relation \eqref{eq_bx=lambdax}, we get $\lambda = r_{k_0(\text{mod} \ l)+1}$. Using the relation \eqref{eq_bx=lambdax} for $k= k_0 + l$, we have
\begin{align}\label{eq_point2}
s_{(k_0 + l-1)(\text{mod} \ l')+1}x_{k_0 + l-1}+r_{(k_0 + l)(\text{mod}\ l)+1}x_{k_0 + l} &= \lambda x_{k_0 + l}.\nonumber\\
\text{That is,}\ \ \ s_{(k_0 + l-1)(\text{mod} \ l')+1}x_{k_0 + l-1}+r_{k_0 (\text{mod}\ l)+1}x_{k_0 + l}& = \lambda x_{k_0 + l}.
\end{align}
Putting $\lambda = r_{k_0(\text{mod} \ l)+1}$ in the Equation \eqref{eq_point2}, we get $s_{(k_0 +l-1)(\text{mod} \ l')+1}x_{k_0 +l-1}=0$. As $s_{(k_0 +l-1)(\text{mod} \ l')+1}\neq 0$, therefore $x_{k_0 +l-1}=0$. Similarly, using the relation \eqref{eq_bx=lambdax} for $k=k_0 + l-1$ and putting the value $x_{k_0 +l-1}=0$, we obtain $x_{k_0 +l-2}=0$. Repeating the same steps for $k=k_0 + l-2, k_0 + l-3, \dots ,k_0 + 1$, we get $x_{k_0}=0$, which is a contradiction. Hence, $\sigma_p(\mathbb{B},c_0)=\phi$.
\end{proof}
Let $B^*$ denote the adjoint of the operator $B$. Then the matrix representation of $B^*$ is equal to the transpose of the matrix $B$, which is given follows:
\begin{align*}
B^*&=
\left[
\begin{array}{ccccccc}
r_1  & s_1 & 0 & 0 & 0& \cdots\\
0 & \ddots & s_2& 0 & 0 &\cdots\\
0&0  & r_l &\ddots & 0 & \cdots\\
0& 0& 0& \ddots & s_{l'}& \cdots\\
\vdots &\vdots &\vdots &\vdots &\ddots & \ddots
\end{array}
\right].
\end{align*}
The next theorem gives the point spectrum of the operator $B^*$.
\begin{theorem}\label{th_pointspectrumofadjoint}
$\sigma_p(\mathbb{B}^*, c_0^*) = \left\{\lambda \in \mathbb{C}: (\vert \lambda - r_1\vert \dots \vert \lambda - r_l\vert)^{\frac{1}{l}} < (\vert s_1 \vert \dots \vert s_{l'}\vert)^{\frac{1}{l'}}\right\}$.
\end{theorem}
\begin{proof}
Let $\lambda \in \sigma_p(\mathbb{B}^*, c^*_0 \cong l_1)$. Then there exists a nonzero sequence $x=(x_k) \in l_1$ such that $\mathbb{B}^*x = \lambda x$. From this relation, the subsequences $(x_{kL}), (x_{kL+1}), \dots, (x_{kL+L-1})$ of $x=(x_k)$ are given by
\begin{align*}
x_{kL}=& \left\{ \frac{((\lambda - r_1) \dots (\lambda - r_l))^{\frac{L}{l}}}{(s_1 \dots s_{l'})^{\frac{L}{l'}}}\right\}^k x_0\\
x_{kL+1}=& \frac{(\lambda - r_1)}{s_1}\left\{ \frac{((\lambda - r_1) \dots (\lambda - r_l))^{\frac{L}{l}}}{(s_1 \dots s_{l'})^{\frac{L}{l'}}}\right\}^k x_0\\
\vdots&\\
x_{kL+L-1}=&\frac{(\lambda - r_1)^{\frac{L}{l}} \dots (\lambda - r_{l-1})^{\frac{L}{l}}(\lambda -r_l)^{\frac{L}{l}-1}}{s_1^{\frac{L}{l'}}\dots s_{l' -1}^{\frac{L}{l'}}s_{l'}^{\frac{L}{l'}-1}}\left\{ \frac{((\lambda - r_1) \dots (\lambda - r_l))^{\frac{L}{l}}}{(s_1 \dots s_{l'})^{\frac{L}{l'}}}\right\}^k x_0.
\end{align*}
Thus,
\begin{align*}
\sum\limits_{k=0}^{\infty}\lvert x_k \rvert &= \sum\limits_{k=0}^{\infty} \lvert x_{kL}\rvert + \sum\limits_{k=0}^{\infty} \lvert x_{kL+1}\rvert + \dots + \sum\limits_{n=0}^{\infty} \lvert x_{kL+L-1}\rvert\\
 &= \left( 1+  \left \lvert \frac{\lambda - r_1}{s_1}\right\rvert + \cdots +\left\lvert  \frac{(\lambda - r_1)^{\frac{L}{l}} \cdots (\lambda - r_{l-1})^{\frac{L}{l}}(\lambda -r_l)^{\frac{L}{l}-1}}{s_1^{\frac{L}{l'}}\cdots s_{l' -1}^{\frac{L}{l'}}s_{l'}^{\frac{L}{l'}-1}}\right\rvert \right)\\
& \qquad \times \sum\limits_{k=0}^{\infty}\left\lvert \frac{((\lambda - r_1) \cdots (\lambda - r_l))^{\frac{L}{l}}}{(s_1 \cdots s_{l'})^{\frac{L}{l'}}}   \right\rvert^k \lvert x_0 \rvert.
\end{align*}
Clearly, the sequence $x=(x_k)\in l_1$ if and only if $(\vert \lambda - r_1\vert \cdots \vert \lambda - r_l\vert)^{\frac{1}{l}} < (\vert s_1 \vert \cdots \vert s_{l'}\vert)^{\frac{1}{l'}}$. This proves the theorem.
\end{proof}

\begin{theorem}\label{th_residualspectrum}
$\sigma_r(\mathbb{B}, c_0) = \left\{\lambda \in \mathbb{C}: (\vert \lambda - r_1\vert \cdots \vert \lambda - r_l\vert)^{\frac{1}{l}} < (\vert s_1 \vert \cdots \vert s_{l'}\vert)^{\frac{1}{l'}}\right\}$.
\end{theorem}
\begin{proof}
The residual spectrum of a bounded linear operator $T$ on a Banach space $X$ is given by the relation $\sigma_r(T, X) = \sigma_p(T^*, X^*) \setminus \sigma_p(T, X)$. Thus,
 using Theorems \ref{th_pointspectrum} and \ref{th_pointspectrumofadjoint}, we get the result.
\end{proof}

\begin{theorem}\label{th_continuousspectrum}
$\sigma_c(\mathbb{B}, c_0) = \left\{\lambda \in \mathbb{C}: (\vert \lambda - r_1\vert \cdots \vert \lambda - r_l\vert)^{\frac{1}{l}} = (\vert s_1 \vert \cdots \vert s_{l'}\vert)^{\frac{1}{l'}}\right\}$.
\end{theorem}
\begin{proof}
Since spectrum of an operator on a Banach space is disjoint union of point, residual and continuous spectrum, therefore from Theorem \ref{th_pointspectrum}, \ref{th_spectrum} and \ref{th_residualspectrum}, we get $\sigma_c(\mathbb{B}, c_0) = \left\{\lambda \in \mathbb{C}: (\vert \lambda - r_1\vert \cdots \vert \lambda - r_l\vert)^{\frac{1}{l}} = (\vert s_1 \vert \cdots \vert s_{l'}\vert)^{\frac{1}{l'}}\right\}$.
\end{proof}

\begin{theorem}
$\{r_1, r_2, \dots, r_l\}\subseteq III_1\sigma(\mathbb{B}, c_0)$.
\end{theorem}
\begin{proof}
 The range of the operator $\mathbb{B}- r_1 I$ is not dense in $c_0$, therefore $r_1 \in III\sigma(\mathbb{B}, c_0)$. Since $\lambda = r_1$, the inequality $ (\vert \lambda - r_1\vert \cdots \vert \lambda - r_l\vert)^{\frac{1}{l}} < (\vert s_1 \vert \cdots \vert s_{l'}\vert)^{\frac{1}{l'}}$ holds. Hence, from Theorem \ref{th_residualspectrum}, we get $r_1 \in \sigma_r(\mathbb{B}, c_0)$. However, $\sigma_r (\mathbb{B}, c_0) = III_1\sigma(\mathbb{B}, c_0) \cup III_2\sigma(\mathbb{B}, c_0)$. Therefore, $r_1$ belongs to either $III_1\sigma(\mathbb{B}, c_0)$ or $ III_2\sigma(\mathbb{B}, c_0)$. To prove $r_1 \in III_1\sigma(\mathbb{B},c_0)$, we show that the matrix $\mathbb{B}-r_1 I$ has bounded inverse and from Theorem \ref{th_adjointontobounded}, it will be sufficient to show that $(\mathbb{B}-\lambda I)^ *$ is onto. For this, let $y=(y_k)\in l_1$. Then $(\mathbb{B}-\lambda I)^* x = y$ implies
 \begin{equation*}
 x_{mL+1}=k_{11}y_{mL},
 \end{equation*}
  where $k_{11}=s_1^{-1}$,
  \begin{equation*}
  x_{mL+2}=k_{21}y_{mL+1}+k_{22}y_{mL},
  \end{equation*}
  where $k_{21}= (s_{1(\text{mod} \ l')+1})^{-1}$ and $k_{22}= - (r_{1(\text{mod} \ l)+1}-r_1)(s_1s_{1(\text{mod} \ l')+1})^{-1}$. Similarly, we get
  \begin{align*}
  x_{mL+3}&= k_{31}y_{mL+2} + k_{32}y_{mL+1}\\
  &\vdots\\
  x_{mL+L-1}&= k_{L-1,1}y_{mL+L-2}+k_{L-1,2}y_{mL+l-2}\\
  x_{mL}&= k_{L1}y_{mL+L-1}+k_{L2}y_{mL+L-2}.
  \end{align*}
Hence,
\allowdisplaybreaks
\begin{align*}
\sum\limits_{n=0}^{\infty}\lvert x_n \rvert  &= \lvert x_0 \rvert + \sum\limits_{m=1}^{\infty}\lvert x_{mL}\rvert + \sum\limits_{m=0}^{\infty}\lvert x_{mL+L-1}\rvert + \dots +\sum\limits_{m=0}^{\infty}\lvert x_{mL+1}\rvert \\
& \leq \lvert x_0 \rvert + \vert k_{L1}\vert\sum\limits_{m=0}^{\infty}\vert y_{mL + L-1}\vert + (\vert k_{L2}\vert+\vert k_{L-1,1}\vert)\sum\limits_{m=0}^{\infty}y_{mL + L-2}\\
&\qquad+ \dots + (\vert k_{32}\vert + \vert k_{21}\vert)\sum\limits_{m=0}^{\infty}\vert y_{mL+1}\vert + (\vert k_{22}\vert+\vert k_{11}\vert)\sum\limits_{m=0}^{\infty}\vert y_{mL}\vert.
\end{align*}
Since the sequence $y=(y_k)\in l_1$, therefore from the above inequality, we have $x=(x_k)\in l_1$. Thus, $(\mathbb{B}-r_1 I)^*$ is onto and therefore by theorem \ref{th_adjointontobounded} the operator $(\mathbb{B}-r_1 I)$ has bounded inverse. Hence, $r_1 \in  III_1\sigma(\mathbb{B}(r_1,r_2; s_1, s_2, s_3), c_0)$. Analogously, we prove that $r_2, \dots, r_l \in III_1\sigma(\mathbb{B}(r_1,r_2; s_1, s_2, s_3), c_0)$. This proves the theorem.
\end{proof}

\begin{theorem}
 $\sigma_r (\mathbb{B}, c_0)\setminus\{r_1, r_2 \dots r_l\}\subseteq III_2\sigma(\mathbb{B}, c_0)$.
\end{theorem}
\begin{proof}
Let $\lambda \in \sigma_r (\mathbb{B}, c_0)\setminus\{r_1, r_2 \dots r_l\}$. Then $\lambda \notin r_1, \dots, \lambda \notin r_l$ and table \ref{table1} suggests that $\lambda$ belongs to either $III_1\sigma(\mathbb{B}, c_0)$ or $III_2 \sigma(\mathbb{B}, c_0)$. To prove $ \lambda \in III_2 \sigma(\mathbb{B},c_0)$, it is enough to show that the operator $(\mathbb{B}- \lambda I)$ does not have bounded inverse. Since $\lambda \in \sigma_r(\mathbb{B}, c_0)$, the inequality $(\vert \lambda - r_1\vert \dots \vert \lambda - r_l\vert)^{\frac{1}{l}} < (\vert s_1 \vert \dots \vert s_{l'}\vert)^{\frac{1}{l'}}$ holds. Therefore, the series $\sum_{k=0}^{\infty}\vert z_{nk}\vert$ in \eqref{eq_inverse} is not convergent as $n$ goes to infinity and hence the operator $\mathbb{B}-\lambda I$ has no bounded inverse. Thus, $\lambda \in III_2\sigma(\mathbb{B},c_0)$. This proves the theorem.
\end{proof}
\section{Spectra and fine spectra of the matrix $\mathbb{B}(r_1, \dots, r_4; s_1, \dots, s_6)$}
We consider the matrix
\begin{align*}
&\mathbb{B}(r_1, \dots, r_4; s_1, \dots, s_6)\\
&=
  \left[{\begin{array}{cccccccccccccc}
    r_1 & 0 & 0 & 0 & 0&0 &0&0&0&0&0&0&0& \dots\\
    s_1 & r_2 & 0 & 0 & 0&0&0&0&0&0&0&0&0&\dots\\
    0 & s_2 & r_3 & 0 & 0&0&0&0&0&0&0&0&0&\dots\\
    0 & 0 & s_3 & r_4 & 0&0&0&0&0&0&0&0&0&\dots\\
    0 & 0& 0 & s_4& r_1&0&0&0&0&0&0&0&0&\dots\\
    0 & 0 & 0 & 0 & s_5&r_2&0&0&0&0&0&0&0&\dots\\
    0 & 0 & 0 & 0 & 0&s_6&r_3&0&0&0&0&0&0&\dots\\
    0 & 0 & 0 & 0 & 0&0&s_1&r_4&0&0&0&0&0&\dots\\
    0 & 0 & 0 & 0 & 0&0&0&s_2&r_1&0&0&0&0&\dots\\
    0 & 0 & 0 & 0 & 0&0&0&0&s_3&r_2&0&0&0&\dots\\
    0 & 0 & 0 & 0 & 0&0&0&0&0&s_4&r_3&0&0&\dots\\
    0 & 0 & 0 & 0 & 0&0&0&0&0&0&s_5&r_4&0&\dots\\
    0 & 0 & 0 & 0 & 0&0&0&0&0&0&0&s_6&r_1&\dots\\
     \vdots & \vdots & \vdots & \vdots & \vdots&\vdots&\vdots&\vdots&\vdots&\vdots&\vdots&\vdots&\vdots&\ddots
     \end{array}}\right].
\end{align*}
Now, consider the following sets:
\begin{align*}
S&=\left\{ \lambda \in \mathbb{C}: (\vert \lambda - r_1\vert \vert \lambda - r_2\vert \vert \lambda-r_3\vert \vert \lambda - r_4\vert)^{\frac{1}{4}}
\leq (\vert s_1 \vert \vert s_2\vert \vert s_3\vert \vert s_4\vert \vert s_5\vert \vert s_{6}\vert)^{\frac{1}{6}}\right\},\\
S_1&=\left\{ \lambda \in \mathbb{C}: (\vert \lambda - r_1\vert \vert \lambda - r_2\vert \vert \lambda-r_3\vert \vert \lambda - r_4\vert)^{\frac{1}{4}}
 < (\vert s_1 \vert \vert s_2\vert \vert s_3\vert \vert s_4\vert \vert s_5\vert \vert s_{6}\vert)^{\frac{1}{6}}\right\},\\
S_2&=\left\{ \lambda \in \mathbb{C}: (\vert \lambda - r_1\vert \vert \lambda - r_2\vert \vert \lambda-r_3\vert \vert \lambda - r_4\vert)^{\frac{1}{4}}
= (\vert s_1 \vert \vert s_2\vert \vert s_3\vert \vert s_4\vert \vert s_5\vert \vert s_{6}\vert)^{\frac{1}{6}}\right\}.
\end{align*}
From the discussion of the previous section, we get the following results:
\begin{corollary}
The operator $\mathbb{B}(r_1, \dots, r_4; s_1, \dots, s_6): c_0 \rightarrow c_0$ is a bounded linear operator and $\lVert \mathbb{B}(r_1, \dots, r_4; s_1, \dots, s_6)\rVert_{(c_0:c_0)} \leq \max\limits_{i,j}\{\lvert r_i \rvert + \lvert s_j \lvert; 1\leq i \leq 4, 1 \leq j \leq 6\}$.
\end{corollary}
\begin{theorem}
$\sigma(\mathbb{B}(r_1, \dots, r_4; s_1, \dots, s_6),c_0)=S$.
\end{theorem}

\begin{theorem}
$\sigma_p(\mathbb{B}(r_1, \dots, r_4; s_1, \dots, s_6),c_0)=\phi$.
\end{theorem}

\begin{theorem}
$\sigma_p(\mathbb{B}(r_1, \dots, r_4; s_1, \dots, s_6)^*,c_0^*\cong l_1)=S_1$.
\end{theorem}

\begin{theorem}
$\sigma_r(\mathbb{B}(r_1, \dots, r_4; s_1, \dots, s_6),c_0)=S_1$.
\end{theorem}

\begin{theorem}
$\sigma_c(\mathbb{B}(r_1, \dots, r_4; s_1, \dots, s_6),c_0)=S_2$.
\end{theorem}

\begin{theorem}
$\{r_1, r_2, r_3, r_4\} \subseteq III_1\sigma(\mathbb{B}(r_1, \dots, r_4; s_1, \dots, s_6),c_0)$.
\end{theorem}

\begin{theorem}
 $S_1\setminus \{r_1, r_2, r_3, r_4\}\subseteq III_2\sigma(\mathbb{B}(r_1, \dots, r_4; s_1, \dots, s_6),c_0)$.
\end{theorem}

In particular, if we take $r_1=1-i,$ $r_2=-i,$ $r_3=-1.5,$ $r_4=-i$ and $s_1= i,$ $s_2= 1+i,$ $s_3= -2,$ $s_4= -1.5,$ $s_5= 1-i,$ $s_6= -1$, then the spectrum is given by
\[\sigma(\mathbb{B}(r_1, \dots, r_4; s_1, \dots, s_6),c_0)=\left\{ \lambda \in \mathbb{C}: (\vert \lambda - 1+i\vert \vert \lambda +i\vert^2 \vert \lambda+1.5\vert)^{\frac{1}{4}} \leq 6^{\frac{1}{6}}\right\},\]
which is shown by the shaded region in Figure 1.

\begin{figure}\label{fig1}
\begin{center}
\includegraphics[scale=0.9]{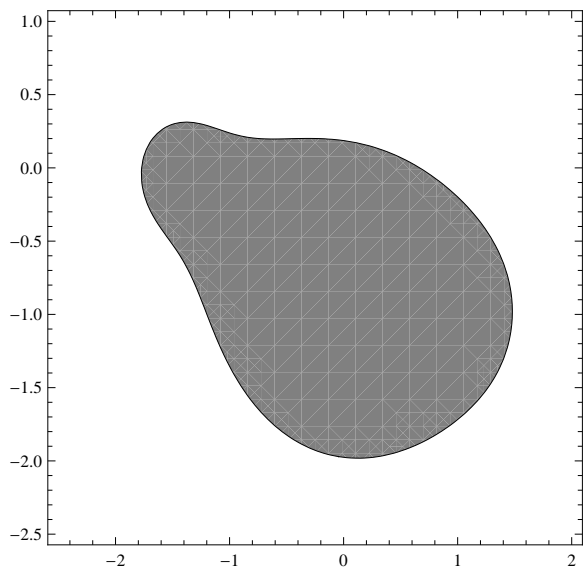}
\caption{Spectrum of $\mathbb{B}(r_1, \dots, r_4; s_1, \dots, s_6)$.}
\end{center}
\end{figure}

\bibliographystyle{plain}
\bibliography{mybib}

\begin{thebibliography}{10}

\bibitem{akhmedov2007fine}
A.~M. Akhmedov and F.~Ba{\c{s}}ar.
\newblock The fine spectra of the difference operator $\delta$ over the
  sequence space $bv_p, (1\leq p<\infty)$.
\newblock {\em Acta Mathematica Sinica}, 23(10):1757--1768, 2007.

\bibitem{akhmedov2008fine}
A.~M. Akhmedov and F.~Ba{\c{s}}ar.
\newblock The fine spectra of the \text{C}es{\`a}ro operator $\text{C}_1$ over
  the sequence space $bv_p,(1 \leq p< \infty)$.
\newblock {\em Math. J. of Okayama Univ.}, 50(1):135--147, 2008.

\bibitem{altay2004fine}
B.~Altay and F.~Ba{\c{s}}ar.
\newblock On the fine spectrum of the difference operator $\delta$ on $c_0$ and
  $c$.
\newblock {\em Information Sciences}, 168(1):217--224, 2004.

\bibitem{altay2005fine}
B.~Altay and F.~Ba{\c{s}}ar.
\newblock On the fine spectrum of the generalized difference operator
  \textit{B}$(r, s)$ over the sequence spaces $c_0$ and $c$.
\newblock {\em Int. J. Math. Math. Sci.}, 2005(18):3005--3013, 2005.

\bibitem{bilgicc2008fine}
H.~Bilgi{\c{c}} and H.~Furkan.
\newblock On the fine spectrum of the generalized difference operator
  \textit{B}$(r, s)$ over the sequence spaces $l_p$ and $bv_p, (1< p<\infty)$.
\newblock {\em Nonlinear Anal.}, 68(3):499--506, 2008.

\bibitem{birbonshi2017some}
R.~Birbonshi and P.~D. Srivastava.
\newblock On some study of the fine spectra of $n$-th band triangular matrices.
\newblock {\em Complex Anal. Oper. Theory}, 11(4):739--753, 2017.

\bibitem{goldberg2006unbounded}
S.~Goldberg.
\newblock {\em Unbounded linear operators: Theory and applications}.
\newblock Courier Corporation, 2006.

\bibitem{gonzalez1985fine}
M.~Gonzalez.
\newblock The fine spectrum of the \text{C}es{\`a}ro operator in $l_p (1<
  p<\infty)$.
\newblock {\em Arch. Math.}, 44(4):355--358, 1985.

\bibitem{kreyszig1989introductory}
E.~Kreyszig.
\newblock {\em Introductory functional analysis with applications}, volume~1.
\newblock Wiley New York, 1989.

\bibitem{panigrahi2012spectrum}
B.~L. Panigrahi and P.~D. Srivastava.
\newblock Spectrum and fine spectrum of generalized second order difference
  operator $\delta^2_{uv}$ on sequence space $c_0$.
\newblock {\em Thai Journal of Mathematics}, 9(1):57--74, 2011.

\bibitem{panigrahi2012spectrum2}
B.~L. Panigrahi and P.~D. Srivastava.
\newblock Spectrum and fine spectrum of generalized second order forward
  difference operator $\delta^2_{uvw}$ on sequence space $l_1$.
\newblock {\em Demonstratio Mathematica}, 45(3):593--609, 2012.

\bibitem{patra2017some}
A.~Patra, R.~Birbonshi, and P.~D. Srivastava.
\newblock On some study of the fine spectra of triangular band matrices.
\newblock {\em Complex Anal. Oper. Theory}, pages 1--21, 2017.

\bibitem{reade1985spectrum}
J.~B. Reade.
\newblock On the spectrum of the \text{C}es{\`a}ro operator.
\newblock {\em Bull. London Math. Soc.}, 17(3):263--267, 1985.

\bibitem{rhoades1983fine}
B.~E. Rhoades.
\newblock The fine spectra for weighted mean operators.
\newblock {\em Pacific Journal of Mathematics}, 104(1):219--230, 1983.

\bibitem{rhoades1989fine}
B.~E. Rhoades.
\newblock The fine spectra for weighted mean operators in \text{B}$(l^p)$.
\newblock {\em Integral equations and operator theory}, 12(1):82--98, 1989.

\bibitem{srivastava2012fine}
P.~D. Srivastava and S.~Kumar.
\newblock Fine spectrum of the generalized difference operator $\delta_v$ on
  sequence space $l_1$.
\newblock {\em Thai Journal of Mathematics}, 8(2):221--233, 2010.

\bibitem{srivastava2012fine2}
P.~D. Srivastava and S.~Kumar.
\newblock Fine spectrum of the generalized difference operator $\delta_{uv}$ on
  sequence space $l_1$.
\newblock {\em Appl. Math. Comput.}, 218(11):6407--6414, 2012.

\bibitem{stieglitz1977matrixtransformationen}
Michael Stieglitz and Hubert Tietz.
\newblock Matrixtransformationen von folgenr{\"a}umen eine
  ergebnis{\"u}bersicht.
\newblock {\em Mathematische Zeitschrift}, 154(1):1--16, 1977.

\bibitem{taylor1957general}
A.~E. Taylor and Charles J.~A. Halberg~Jr.
\newblock General theorems about a bounded linear operator and its conjugate.
\newblock {\em J. Reine Angew. Math}, 198:93--111, 1957.

\bibitem{wenger1975fine}
R.~B. Wenger.
\newblock The fine spectra of the \text{H}{\"o}lder summability operators.
\newblock {\em Indian J. Pure Appl. Math}, 6(6):695--712, 1975.

\bibitem{yildirim1996spectrum}
M.~Yildirim.
\newblock On the spectrum and fine spectrum of the compact rhaly operators.
\newblock {\em Indian Journal of Pure and Applied Mathematics}, 27(8):779--784,
  1996.

\end{thebibliography}

\end{document}